\documentclass[amsthm]{elsart}

\usepackage{yjsco}
\usepackage{natbib}
\usepackage{amssymb}
\usepackage{amsmath}

\def\gf#1{\mathbb F_{#1}}
\def\coc#1{\mathcal C(#1)}  
\def\mcoc#1{\mathcal M(#1)} 
\def\cob#1{\mathcal B(#1)}  
\def\comp#1{\mathcal D(#1)} 

\numberwithin{equation}{section}

\begin{document}
\begin{frontmatter}

\title{The Moufang Loops of Order 64 and 81}

\thanks{Petr Vojt\v{e}chovsk\'y supported by the PROF 2006 grant of the University of Denver.}

\author{G\'abor P.~Nagy}

\address{Bolyai Institute, University of Szeged, Aradi vertanuk tere 1, 6720
Szeged, Hungary}

\ead{nagyg@math.u-szeged.hu}

\author{Petr Vojt\v{e}chovsk\'y}

\address{Department of Mathematics, University of Denver,
Denver, Colorado 80208, U.S.A.}

\ead{petr@math.du.edu}

\begin{abstract}
We classify Moufang loops of order 64 and 81 up to isomorphism, using a linear
algebraic approach to central loop extensions. In addition to the 267 groups of
order 64, there are 4262 nonassociative Moufang loops of order 64. In addition
to the $15$ groups of order $81$, there are $5$ nonassociative Moufang loops of
order $81$, $2$ of which are commutative.
\end{abstract}


\end{frontmatter}


\section{Introduction}

\subsection{Mathematical background}

Let $Q$ be a set with one binary operation, denoted by juxtaposition. For $x\in
Q$, define the \emph{left translation} $L_x$ and the \emph{right translation}
$R_x$ by $L_x(y) = xy$, $R_x(y) = yx$. Then $Q$ is a \emph{loop} if all
translations are bijections of $Q$, and if $Q$ possesses a \emph{neutral
element} $1$ satisfying $1x = x = x1$ for every $x\in Q$.

A loop $Q$ is \emph{Moufang} if it satisfies the Moufang identity $(xy)(zx) =
x((yz)x)$. Although Moufang loops are not associative in general, they have
many properties we are familiar with from the theory of groups. For instance,
every element $x$ of a Moufang loop is paired with its inverse $x^{-1}$
satisfying $x^{-1}(xy) = y = (yx)x^{-1}$, any two elements generate a subgroup
(this property is called \emph{diassociativity}), and any three elements that
associate generate a subgroup (the famous Moufang theorem).

Indeed, the fact that the methods used in this paper work for Moufang loops can
be seen as another confirmation of their proximity to groups.

The \emph{center} $Z(Q)$ of a loop $Q$ consist of all elements that commute and
associate with all other elements of $Q$. A subloop $S$ of a loop $Q$ is a
nonempty subset of $Q$ that happens to be a loop with respect to the
multiplication inherited from $Q$. To see whether a subset $S\ne\emptyset$ of a
Moufang loop $Q$ is a subloop of $Q$, it suffices to check that $S$ is closed
under multiplication and inverses.

A subloop $S$ of a loop $Q$ is \emph{normal} in $Q$ if $S$ is invariant under
all inner maps $R_{xy}^{-1}R_yR_x$, $L_{yx}^{-1}L_yL_x$ and $L_x^{-1}R_x$. The
center $Z(Q)$ is a normal subloop of $Q$.

A loop $Q$ is \emph{centrally nilpotent} if the sequence
\begin{displaymath}
    Q,\,Q/Z(Q),\, (Q/Z(Q))/Z(Q/Z(Q)),\, \dots
\end{displaymath}
eventually yields the trivial loop.

Loops of order $p^k$, $p$ a prime, are known as $p$-loops. A finite Moufang
loop $Q$ is a $p$-loop if and only if every element of $Q$ has order that is a
power of $p$.

Let $Q$ be a centrally nilpotent $p$-loop. The \emph{Frattini subloop}
$\Phi(Q)$ of $Q$ is the intersection of all maximal subloops of $Q$, or,
equivalently, the subloop consisting of all non-generators of $Q$. Then
$Q/\Phi(Q)$ is an elementary abelian $p$-group and $|Q/\Phi(Q)|=p^d$, where $d$
is the size of a smallest generating set of $Q$, by \citep[Theorem 2.3]{Br}.

An \emph{isotopism} of loops $Q_1$, $Q_2$ is a triple $(\alpha,\beta,\gamma)$
of bijections $Q_1\to Q_2$ such that $\alpha(x)\beta(y) = \gamma(xy)$ holds for
every $x$, $y\in Q_1$. Then $Q_2$ is an \emph{isotope} of $Q_1$.

Clearly, when two loops are isomorphic, they are also isotopic. The converse is
not true in general. If all isotopes of a loop $Q$ are already isomorphic to
$Q$, we call $Q$ a \emph{G-loop} (since groups have this property). Moufang
$2$-loops are G-loops, and that is why a classification of such loops up to
isomorphism is also a classification up to isotopism.

We refer the reader to \citep{Br} and \citep{Pf} for a systematic introduction
to the theory of loops.

\subsection{Historical background}\label{Ss:HistBackground}

The classification of Moufang loops started in earnest with the work of Chein.
In \citep{Ch}, he described all Moufang loops of order less than $64$. Chein's
results are conveniently summarized in \citep{Go}, and the respective loops are
accessible in electronic form in \citep{LOOPS}. Table \ref{Tb:LessThan64} gives
the number of pairwise nonisomorphic nonassociative Moufang loops of order $n$
for every $1\le n\le 63$ for which at least one nonassociative Moufang loop
exists.

\begin{table}[h]
\caption{The number $M(n)$ of nonassociative Moufang loops of order $n$ less
than 64.}\label{Tb:LessThan64}
\begin{displaymath}
\begin{array}{c|ccccccccccccccc}
    n&12&16&20&24&28&32&36&40&42&44&48&52&54&56&60\\
    \hline
    M(n)&1&5&1&5&1&71&4&5&1&1&51&1&2&4&5
\end{array}
\end{displaymath}
\end{table}

Certain extensions that proved useful in group theory, see \citep{Dr}, were
used by the second author in \citep{Vo} to construct 4262 nonassociative
Moufang loops of order 64. It was also conjectured in \citep{Vo} that no
additional nonassociative Moufang loops of order 64 exist.

The related question ``\emph{For which integers $n$ there exists a
nonassociative Moufang loop of order $n$?}'' has been studied extensively but
is not fully resolved.

By \citep{CheinRajah}, a nonassociative Moufang loop of order $2m$ exists if
and only if a nonabelian group of order $m$ exists. Hence, a nonassociative
Moufang loop of order $2^k$ exists if and only if $k>3$, and for every odd
$m>1$ there is a nonassociative Moufang loop of order $4m$. Here is the case
$2m$, $m$ odd:

\begin{thm}[Chein and Rajah, 2003, Corollary 2.4] Every Moufang loop of order
$2m$, $m>1$ odd, is associative if and only if $m=p_1^{\alpha_1}\cdots
p_k^{\alpha_k}$, where $p_1<\cdots <p_k$ are odd primes and where
\begin{enumerate}
\item[(i)] $\alpha_i\le 2$, for all $i=1$, $\dots$, $k$, \item[(ii)]
$p_j\not\equiv 1\pmod{p_i}$, for any $i$ and $j$, \item[(iii)] $p_j^2\not\equiv
1\pmod{p_i}$, for any $i$ and any $j$ with $\alpha_j=2$.
\end{enumerate}
\end{thm}

Concerning odd orders, we have:

\begin{thm}[Leong and Rajah, 1997] Every Moufang loop of order $p^\alpha
q_1^{\alpha_1}\cdots q_k^{\alpha_k}$ is associative if $p<q_1<\dots<q_k$ are
odd primes, and if one of the following conditions holds:
\begin{enumerate}
\item[(i)] $\alpha\le 3$ and $\alpha_i\le 2$, \item[(ii)] $p\ge 5$, $\alpha\le
4$, and $\alpha_i\le 2$.
\end{enumerate}
\end{thm}

In \citep{Rajah}, Rajah showed that for odd primes $p<q$ a nonassociative
Moufang loop of order $pq^3$ exists if and only if $q\equiv 1\pmod p$. It is
well-known that there are nonassociative Moufang loops of order $3^4$. Indeed,
smallest nonassociative commutative Moufang loops are of order $3^4$---see
\citep{Bol} for the first example attributed to Zassenhaus, and \citep{KN} for
the second nonassociative commutative Moufang loop of order $3^4$. Wright
\citep{Wright} constructed a nonassociative Moufang loop of order $p^5$ for
every prime $p$. Coming back to the classification results, it is shown in
\citep{NagyValsecchi} that there are precisely $4$ nonassociative Moufang loops
of order $p^5$ for every prime $p\ge 5$.

\subsection{Main result}

In this paper we verify computationally:

\begin{thm}\label{Th:Main}
There are $4262$ pairwise nonisomorphic nonassociative Moufang loops of order
$64$.
\end{thm}

\begin{thm}\label{Th:Main81}
There are $5$ pairwise nonisomorphic nonassociative Moufang loops of order
$81$, $2$ of which are commutative. All $5$ of these loops are isotopes of the
$2$ commutative ones.
\end{thm}

Here is our strategy, completely different from that of \citep{Vo}:

Every Moufang loop $Q$ of order $p^{k+1}$, $p$ a prime, is a central extension
of a Moufang loop $K$ of order $p^k$ by the $p$-element field $\gf{p}$, by
Corollary \ref{Cr:ExtWorks}. Moreover, if $K$ is at most two-generated, then
$Q$ is associative, by Proposition \ref{Pr:UselessGroups}. Therefore, in order
to determine all nonassociative Moufang loops of order $p^{k+1}$ one only needs
to consider all central extensions of at least three-generated Moufang loops
$K$ of order $p^k$ by $\gf{p}$.

Each such extension is determined by a Moufang cocycle, a map $K\times
K\to\gf{p}$ satisfying certain cocycle identities (\ref{Eq:Cocycle}),
(\ref{Eq:MoufangCocycle}). All cocycles $K\times K\to\gf{p}$ form a vector
space $\coc{K}$ of dimension $p^{2k} - 2 p^k + 1$, and the Moufang cocycles
form a subspace $\mcoc{K}$ of $\coc{K}$.

Let $\cob{K}$ be the subspace of coboundaries $K\times K\to \gf{p}$, as defined
in (\ref{Eq:Cobound}). Every coboundary is a Moufang cocycle, by Lemma
\ref{Lm:Cobound2}, so $\mcoc{K}$ decomposes as $\cob{K}\oplus\comp{K}$ for some
$\comp{K}$. The system of linear equations whose solution determines $\mcoc{K}$
has about $p^{3k}$ equations in $p^{2k}$ variables. The subspace $\cob{K}$ can
be constructed directly, and its dimension can be determined by means of
generators of $K$, cf. Lemma \ref{Lm:Gens}.

Since two cocycles that differ by a coboundary give rise to isomorphic loops,
by Lemma \ref{Lm:Cobound}, the study of central Moufang extensions of $K$ by
$\gf{p}$ reduces to the study of the vector space $\comp{K}$.

When the dimension of $\comp{K}$ is small, it is possible to calculate all
cocycles of $\comp{K}$, to construct the corresponding extensions, and to test
the resulting Moufang loops for isomorphism. (The isomorphism test is a
nontrivial problem, but the \emph{ad hoc} invariants used in the \textsf{LOOPS}
package prove sufficient here. See \citep{Vo} for a brief description of the
invariants used in the isomorphism test.)

Fortunately---and somewhat unexpectedly---the dimension of $\comp{K}$ happens
to be low for every Moufang loop $K$ of order $32$ and $27$, with the exception
of the elementary abelian $2$-group of order $32$. We do not know how to
estimate the dimension of $\comp{K}$ (and hence of $\mcoc{K})$ theoretically.
See Section \ref{Sc:Cocycles} for more.

To speed up the search, we can further reduce the number of cocycles from
$\comp{K}$ that need to be considered by taking advantage of the action of the
automorphism group of $K$ on $\mcoc{K}$, as described in Section \ref{Sc:Auto}.

The troublesome case where $K$ is the elementary abelian group of order $32$
has to be handled separately. Central extensions of elementary abelian
$2$-groups by $\gf{2}$ are known as \emph{code loops}---a well-studied variety
of Moufang loops with a rich interplay between the associator map, the
commutator map, and the squaring map. We take advantage of this interplay
(combinatorial polarization), and finish the search, as explained in Section
\ref{Sc:Elem}.

\section{Central extensions}

Let $K$, $A$ be loops. Then a loop $Q$ is an \emph{extension} of $K$ by $A$ if
$A$ is a normal subloop of $Q$ such that $Q/A$ is isomorphic to $K$. An
extension $Q$ of $K$ by $A$ is \emph{central} if $A$ is a subloop of $Z(Q)$.

Let us call a map $f:K\times K\to A$ satisfying
\begin{equation}\label{Eq:Cocycle}
    f(1,x) = f(x,1) = 1
\end{equation}
a (\emph{loop}) \emph{cocycle}.

\begin{prop}\label{Pr:Ext1}
Let $K$ be a loop, $A$ an abelian group, and $f:K\times K\to A$ a cocycle.
Define multiplication $*$ on $K\times A$ by
\begin{displaymath}
    (x,a)*(y,b) = (xy,\,abf(x,y)).
\end{displaymath}
Then $Q = (K\times A,*)$ is a loop, in fact a central extension of $K$ by $A$.
\end{prop}
\begin{proof}
It is easy to see that $Q$ is a quasigroup. The cocycle condition
(\ref{Eq:Cocycle}) guarantees that $Q$ has a neutral element, namely $(1,1)$,
and that $1\times A\le Z(Q)$.
\end{proof}

We denote the resulting central extension by $E(K,A,f)$.

The following result belongs to loop-theoretical folklore:

\begin{thm}[Central extensions for loops.]
Let $Q$, $K$ and $A$ be loops such that $A\le Z(Q)$. Then $Q$ is a central
extension of $A$ by $K$ if and only if there is a cocycle $f:K\times K\to A$
such that $Q$ is isomorphic to $E(K,A,f)$.
\end{thm}
\begin{proof}
Note that $A$ is an abelian group because $A\le Z(Q)$. If $Q$ is isomorphic to
$E(K,A,f)$, then Proposition \ref{Pr:Ext1} shows that $Q$ is a central
extension of $K$ by $A$.

Assume that $Q$ is a central extension of $K$ by $A$. Let $\psi:K\to Q/A$ be an
isomorphism, and let $\sigma:K\to Q$ be any map such that $\sigma(x)\in
\psi(x)$ and $\sigma(1)=1$. Then every element of $Q$ can be written uniquely
as $\sigma(x)a$ for some $x\in K$ and $a\in A$. Since $A$ is a central subloop,
we have $(\sigma(x)a)(\sigma(y)b) = (\sigma(x)\sigma(y))(ab)$. As
$\sigma(x)\sigma(y)\in \psi(x)\psi(y) = \psi(xy)$ and $\sigma(xy)\in\psi(xy)$,
we have $(\sigma(x)\sigma(y))(ab) = \sigma(xy)ab f(x,y)$ for some unique
$f(x,y)\in A$. It is now easy to check that the map $f:K\times K\to A$ so
defined is a cocycle.
\end{proof}

Using the Moufang identity $(xy)(zx) = x((yz)x)$, we obtain by straightforward
calculation:

\begin{prop} Let $K$ be a loop, $A$ an abelian group,
and $f:K\times K\to A$ a cocycle. Then $E(K,A,f)$ is a Moufang loop if and only
if $K$ is a Moufang loop and $f$ satisfies
\begin{equation}\label{Eq:MoufangCocycle}
    f(xy,zx) f(x,y) f(z,x) = f(x, (yz)x) f(yz,x) f(y,z)
\end{equation}
for all $x$, $y$, $z\in K$.
\end{prop}

We call a cocycle $f:K\times K\to A$ satisfying (\ref{Eq:MoufangCocycle}) a
\emph{Moufang cocycle}.

It is not necessary to consider all groups while looking for nonassociative
Moufang central extensions:

\begin{prop}\label{Pr:UselessGroups} Let $Q$ be a diassociative loop,
and let $A\le Z(Q)$ be such that $Q/A$ has a generating subset of size at most
$2$. Then $Q$ is a group.
\end{prop}
\begin{proof}
Let $x$, $y\in Q$ be such that $Q/A$ is generated by $\{xA,yA\}$. Let $H$ be
the subloop of $Q$ generated by $\{x,y\}$. Since $Q$ is diassociative, $H$ is a
group.  Moreover, $HA/A$ is a subloop of $Q/A$ containing $\{xA,yA\}$, thus
$HA/A = Q/A$ and $HA=Q$. For $h_1$, $h_2$, $h_3\in H$ and $a_1$, $a_2$, $a_3\in
A$, we have $(h_1a_1)((h_2a_2)(h_3a_3)) = (h_1(h_2h_3))(a_1a_2a_3) =
((h_1h_2)h_3)(a_1a_2a_3) = ((h_1a_1)(h_2a_2))(h_3a_3)$ because $A$ is central
and $H$ is a group. Thus $Q$ is a group.
\end{proof}

Let $K$ be a loop and $A$ an abelian group. Given a map $\tau:K\to A$, denote
by $\delta\tau:K\times K\to A$ the map defined by
\begin{equation}\label{Eq:Cobound}
    \delta\tau(x,y) = \tau(xy)\tau(x)^{-1}\tau(y)^{-1}.
\end{equation}
Observe that $\delta\tau$ is a cocycle if and only if $\tau(1)=1$. We call a
cocycle of the form $\delta\tau$ (necessarily with $\tau(1)=1$) a
\emph{coboundary}.

From now on we denote the operation in the abelian group $A$ additively and let
$0$ be the neutral element of $A$.

\begin{lem}\label{Lm:Cobound}
Let $K$ be a loop, $A$ an abelian group, and $f$, $g:K\times K\to A$ cocycles.
If $g-f$ is a coboundary then $E(K,A,f)$ is isomorphic to $E(K,A,g)$.
\end{lem}
\begin{proof}
Denote the multiplication in $E(K,A,f)$ by $*$, and the multiplication in
$E(K,A,g)$ by $\circ$. Let $g-f=\delta\tau$ for some $\tau:K\to A$. Define
$\psi:E(K,A,f)\to E(K,A,g)$ by $\psi(x,a) = (x,a+\tau(x))$. Then $\psi$ is
clearly one-to-one, and $\psi(x,a-\tau(x)) = (x,a)$ shows that $\psi$ is also
onto. Now,
\begin{multline*}
    \psi((x,a)*(y,b)) =\psi(xy, a+b+f(x,y)) \\
    = (xy,a+b+f(x,y)+\tau(xy))=(xy,a+b+g(x,y)+\tau(x)+\tau(y))\\
    = (x,a+\tau(x))\circ(y,b+\tau(y))=\psi(x,a)\circ\psi(y,b),
\end{multline*}
and we are through.
\end{proof}

The converse of Lemma \ref{Lm:Cobound} is not true in general. We have:

\begin{lem}\label{Lm:Cobound2}
Let $K$ be a Moufang loop, $A$ an abelian group, and $\delta\tau:K\times K\to
A$ a coboundary. Then $\delta\tau$ is a Moufang cocycle.
\end{lem}
\begin{proof}
We need to show that (\ref{Eq:MoufangCocycle}) holds for $\delta\tau$, that is
\begin{displaymath}
    \delta\tau(xy,zx) + \delta\tau(x,y) + \delta\tau(z,x) =
    \delta\tau(x, (yz)x) + \delta\tau(yz,x) +\delta\tau(y,z).
\end{displaymath}
This is equivalent to
\begin{multline*}
    \tau((xy)(zx))-\tau(xy)-\tau(zx)+\tau(xy)-\tau(x)-\tau(y)
    +\tau(zx)-\tau(z)-\tau(x)\\
    =
    \tau(x((yz)x))-\tau(x)-\tau((yz)x)+\tau((yz)x)-\tau(yz)
    -\tau(x)+\tau(yz)-\tau(y)-\tau(z),
\end{multline*}
which holds because $K$ is Moufang.
\end{proof}

\section{Nonequivalent cocycles}\label{Sc:Cocycles}

Let $\gf{p}=\{0,\dots,p-1\}$ be the $p$-element field and $K$ a Moufang loop.
The cocycles $K\times K\to\gf{p}$ form a vector space $\coc{K}$ over $\gf{p}$,
Moufang cocycles form a subspace $\mcoc{K}$ of $\coc{K}$, and coboundaries form
a subspace $\cob{K}$ of $\mcoc{K}$, by Lemma \ref{Lm:Cobound2}.

We say that two cocycles $f$, $g:K\times K\to \gf{p}$ are \emph{equivalent} if
$f-g$ is a coboundary.

Let $n = |K|^2$, and let $b:K\times K \to \{1,\dots,n\}$ be a fixed bijection.
Let $v_1$, $\dots$, $v_n$ be variables. Identify the cocycle $f:K\times K\to
\gf{p}$ with a vector $(f_1$, $\dots$, $f_n)$ of $\gf{p}^n$ by letting
$f_{b(x,y)} = f(x,y)$. An identity for $f$, such as (\ref{Eq:MoufangCocycle}),
can then be translated into a set of equations in $\gf{p}[v_1,\dots,v_n]$.

For instance, the cocycle identities $f(1,x) = 0$, $f(x,1) = 0$ give rise to
the $2|K|-1$ linear equations
\begin{equation}\label{Eq:System1}
    v_{b(1,x)} = 0,\quad v_{b(x,1)} = 0,\quad \textrm{for\ }x\in K.
\end{equation}
Since the equations of (\ref{Eq:System1}) are linearly independent, we have
$\dim(\coc{K}) = |K|^2 - 2|K| + 1$.

The subspace $\cob{K}$ of coboundaries can be described directly. For $1\ne
x\in K$ let $\tau_x:K\to \gf{p}$ be the map
\begin{displaymath}
    \tau_x(y) = \left\{\begin{array}{ll}
        1,&\textrm{if $y=x$},\\
        0,&\textrm{otherwise.}
    \end{array}\right.
\end{displaymath}
Then $\{\tau_x;\;1\ne x\in K\}$ is a basis of the vector space of all maps
$K\to\gf{p}$. Since the operator $\delta:\tau\mapsto\delta\tau$ is linear,
$\cob{K}$ is generated by $\{\delta\tau_x;\;1\ne x\in K\}$, and thus
$\dim(\cob{K})\le |K| - 1$. In fact:

\begin{lem}\label{Lm:Gens}
Let $Q$ be a Moufang $p$-loop of order $p^k$, and let $d$ be the size of a
minimal generating set of $Q$. Let $\cob{Q} = \{\delta\tau;\; \tau:Q\to
\gf{p},\,\tau(1) = 0\}$ be the vector space of coboundaries. Then
$\dim(\cob{Q}) = p^k-1-d$. Equivalently, $|Q/\Phi(Q)| =
p^{p^k-1-\dim(\cob{Q})}$, where $\Phi(Q)$ is the Frattini subloop of $Q$.
\end{lem}
\begin{proof}
We know that $|Q/\Phi(Q)| = p^d$. Note that $\delta:\tau\mapsto\delta\tau$ is a
homomorphism onto $\cob{Q}$ with $\ker{\delta} = \mathrm{Hom}(Q,\gf{p})$.

Consider the map $\psi:\mathrm{Hom}(Q/\Phi(Q),\gf{p}) \to
\mathrm{Hom}(Q,\gf{p})$ defined by
\begin{displaymath}
    \psi(f)(x) = f(x\Phi(Q)).
\end{displaymath}
Then $\psi$ is a monomorphism, and we claim that it is onto. Consider
$f\in\mathrm{Hom}(Q,\gf{p})$. Since $\dim(\gf{p})=1$, $\ker{f}$ is either all
of $Q$ or it is a maximal subloop of $Q$. In any case, $\Phi(Q)\le\ker{f}$.
Then $\overline{f}:Q/\Phi(Q)\to \gf{p}$, $x\Phi(Q)\mapsto f(x)$ is a
well-defined homomorphism, and $\psi(\overline{f})(x) = \overline{f}(x\Phi(Q))
= f(x)$, so $\psi(\overline{f}) = f$.

$Q/\Phi(Q)$ is a vector space of dimension $d$, and thus its dual
$\mathrm{Hom}(Q/\Phi(Q),\gf{p})$ has also dimension $d$. Hence
$\mathrm{Hom}(Q,\gf{p})$ has dimension $d$, by the above paragraph. Altogether,
$\dim(\mathrm{im}\,{\delta}) = \dim(\gf{p}^Q) - \dim(\mathrm{Hom}(Q,\gf{p})) =
p^k - d$. We have $\dim(\cob{Q}) = \dim(\mathrm{im}\,\delta) - 1$ due to the
requirement that every coboundary is of the form $\delta\tau$ for some $\tau$
satisfying $\tau(1)=0$.
\end{proof}

We now determine $\mcoc{K}$. The Moufang cocycle identity
(\ref{Eq:MoufangCocycle}) gives rise to the $|K|^3$ linear equations
\begin{equation}\label{Eq:System2}
    v_{b(xy,zx)} + v_{b(x,y)} + v_{b(z,x)} - v_{b(x,(yz)x)} - v_{b(yz,x)} -
    v_{b(y,z)} = 0,\quad \textrm{for\ }x,\,y,\,z\in K,
\end{equation}
necessarily linearly dependent. The subspace $\mcoc{K}$ of Moufang cocycles is
obtained by solving the system of linear equations (\ref{Eq:System1}) combined
with (\ref{Eq:System2}).

The main reason why Moufang $p$-loops are somewhat amenable to enumeration is
the following result, cf. \citep{Glauberman} and \citep{GW}:

\begin{thm} Moufang $p$-loops are centrally nilpotent.
\end{thm}

In particular:

\begin{cor}\label{Cr:ExtWorks}
A nontrivial Moufang $p$-loop contains a central subgroup of order $p$.
\end{cor}
\begin{proof}
Let $Q$ be a Moufang $p$-loop of order at least $p$. Since $Q$ is centrally
nilpotent, its center $Z(Q)$ is nontrivial. Then $Z(Q)$ is a $p$-group of order
at least $p$, and so it contains an element of order $p$. This element
generates a central subgroup (hence normal subgroup) of order $p$.
\end{proof}

Given a Moufang $p$-loop $K$, choose $\comp{K}$ so that $\mcoc{K} =
\cob{K}\oplus\comp{K}$.

Among the $51$ groups of order $32$, $20$ are two-generated, including the
cyclic group. No nonassociative Moufang loop is two-generated, thanks to
diassociativity. Thus, in order to obtain all Moufang loops of order $64$ up to
isomorphism, it suffices to construct all extensions $E(K,\gf{2},f)$, where $K$
is one of the $71+51-20 = 102$ Moufang loops of order $32$ that are not
two-generated, and where $f$ is chosen from $\comp{K}$.

It is not clear how to estimate the dimension of $\mcoc{K}$ (and hence of
$\comp{K}$) theoretically. Table \ref{Tb:Dimensions} gives the dimensions of
$\mcoc{K}$ and $\cob{K}$ for every Moufang loop of order $32$ that is not
two-generated. The $i$th Moufang loop of order $32$ in the table corresponds to
the $i$th Moufang loop of order $32$ in \citep{Go} and to the $i$th Moufang
loop of order $32$ in \textsf{LOOPS}, where it can be retrieved as
\texttt{MoufangLoop(32,i)}. The $i$th group of order $32$ in the table
corresponds to the $i$th group of order $32$ in \citep{GAP}, where it can be
retrieved as \texttt{SmallGroup(32,i)}.

\begin{table}
\caption{Dimensions of Moufang cocycles $\mcoc{K}$ and coboundaries $\cob{K}$
for all Moufang loops $K$ of order $32$ that are not
two-generated.}\label{Tb:Dimensions}
\begin{small}
\begin{displaymath}
\begin{array}{l|cccccccccccccccc}
\textrm{loop\ $K$}    &1&2&3&4&5&6&7&8&9&10&11&12&13&14&15&16\\
\dim(\mcoc{K})      &41&40&40&36&35&34&35&34&34&40&40&40&40&40&40&40\\
\dim(\cob{K})       &27&27&27&28&28&28&28&28&28&27&27&27&27&27&27&27\\
\hline
\textrm{loop\ $K$}    &17&18&19&20&21&22&23&24&25&26&27&28&29&30&31&32\\
\dim(\mcoc{K})      &40&40&40&40&40&40&34&34&34&34&34&34&34&34&34&34\\
\dim(\cob{K})       &27&27&27&27&27&27&28&28&28&28&28&28&28&28&28&28\\
\hline
\textrm{loop\ $K$}    &33&34&35&36&37&38&39&40&41&42&43&44&45&46&47&48\\
\dim(\mcoc{K})      &34&34&34&34&34&34&34&34&34&33&33&33&33&33&34&34\\
\dim(\cob{K})       &28&28&28&28&28&28&28&28&28&28&28&28&28&28&28&28\\
\hline
\textrm{loop\ $K$}    &49&50&51&52&53&54&55&56&57&58&59&60&61&62&63&64\\
\dim(\mcoc{K})      &33&34&34&33&33&34&34&34&34&33&33&35&34&34&34&34\\
\dim(\cob{K})       &28&28&28&28&28&28&28&28&28&28&28&28&28&28&28&28\\
\hline
\textrm{loop\ $K$}    &65&66&67&68&69&70&71&&&&&&&&&\\
\dim(\mcoc{K})      &34&33&34&34&35&35&34&&&&&&&&&\\
\dim(\cob{K})       &28&28&28&28&28&28&28&&&&&&&&&\\
\hline
\textrm{group\ $K$}   &21&22&23&24&25&26&27&28&29&30&31&32&33&34&35&36\\
\dim(\mcoc{K})      &35&36&35&34&35&34&36&35&34&34&34&33&33&35&34&35\\
\dim(\cob{K})       &28&28&28&28&28&28&28&28&28&28&28&28&28&28&28&28\\
\hline
\textrm{group\ $K$}   &37&38&39&40&41&42&43&44&45&46&47&48&49&50&51&\\
\dim(\mcoc{K})      &34&34&35&34&34&34&34&34&41&41&40&40&40&40&51&\\
\dim(\cob{K})       &28&28&28&28&28&28&28&28&27&27&27&27&27&27&26&\\
\end{array}
\end{displaymath}
\end{small}
\end{table}

Note that for every Moufang loop $K$ listed in Table \ref{Tb:Dimensions} we
have $\dim(\comp{K})\le 14$, with the exception of the elementary abelian group
of order $32$ (the last group in the table).

As for the Moufang loops of order $81$, the only group $K$ of order $27$ that
is not two-generated is the elementary abelian group. We have $\dim(\cob{K}) =
23$ by Lemma \ref{Lm:Gens}, and a short computer calculation yields
$\dim(\mcoc{K}) = 30$. This means that the classification of nonassociative
Moufang loops of order $81$ is an easy task, indeed, for a computer.

\section{Cocycles and the automorphism group}\label{Sc:Auto}

Let $K$ be a loop and $A$ an abelian group. The automorphism group
$\mathrm{Aut}(K)$ acts on $\coc{K}$ by $f\mapsto f^\alpha$, where
$f^\alpha(x,y) = f(\alpha(x),\alpha(y))$.

\begin{lem}\label{Lm:Aut}
Let $K$ be a loop, $A$ an abelian group, $f:K\times K\to A$ a cocycle, and
$\alpha\in\mathrm{Aut}(K)$. Then $E(K,A,f)$ is isomorphic to $E(K,A,f^\alpha)$.
\end{lem}
\begin{proof}
Define $\psi:E(K,A,f^\alpha)\to E(K,A,f)$ by $(x,a)\mapsto (\alpha(x),a)$.
Denote the product in $E(K,A,f^\alpha)$ by $*$ and the product in $E(K,A,f)$ by
$\circ$. Then
\begin{multline*}
    \psi((x,a)*(y,b)) = \psi(xy,a+b+f^\alpha(x,y)) = (\alpha(xy),
        a+b+f^\alpha(x,y))\\
    = (\alpha(x)\alpha(y), a+b+f(\alpha(x),\alpha(y))) =
    (\alpha(x),a)\circ (\alpha(y),b)
     = \psi(x,a)\circ \psi(y,b).
\end{multline*}
Since $\psi$ is clearly a bijection, we are done.
\end{proof}

We can therefore use the action of the automorphism group to reduce the number
of nonequivalent cocycles that need to be taken into consideration. Let us
return to the Moufang case, extending a Moufang loop $K$ by $\gf{p}$.

Set $X=\emptyset$ and $Y=\comp{K}$. Until $Y$ is empty, repeat the following:
Pick $f\in Y$ and insert it into $X$. For every $\alpha\in\mathrm{Aut(K)}$,
calculate $f^\alpha$. Decompose $f^\alpha = g_\alpha+h_\alpha$, where
$g_\alpha\in \cob{K}$ and $h_\alpha\in \comp{K}$. Remove $h_\alpha$ from $Y$,
if possible.

We claim that all extensions of $K$ by $\gf{p}$ are obtained if only cocycles
from $X$ are considered. To see this, note that $E(K,A,f)$ is isomorphic to
$E(K,A,f^\alpha)$ by Lemma \ref{Lm:Aut}, and that $E(K,A,f^\alpha)$ is in turn
isomorphic to $E(K,A,h_\alpha)$, because $f^\alpha - h_\alpha = g_\alpha$ is a
coboundary.

The size of $X$ is often much smaller than the size of $\comp{K}$. For
instance, for $K = \texttt{MoufangLoop(32,1)}$ we have $|X| = 246$ (or about
$1.5$ percent of $|\comp{K}| = 2^{14}$), for $K=\texttt{MoufangLoop(32,71)}$ we
have $|X|=20$ ($31.3$ percent), for $K=\mathbb Z_{32}$ we have $|X|=2$ ($100$
percent), for $K=\texttt{SmallGroup(32,50)}$ we have $|X|=138$ ($1.7$ percent),
and for $K$ the elementary abelian group of order $27$ we have $|X| = 11$
($0.5$ percent).

\section{The elementary abelian case in characteristic two}\label{Sc:Elem}

When $K$ is the elementary abelian group of order $32$, the dimension of
$\comp{K}$ is prohibitively large, equal to $25$. In this section we describe
how this case was handled in the search.

As we have already mentioned, Moufang $2$-loops $Q$ with a central subloop $Z$
of order $2$ such that $V=Q/Z$ is an elementary abelian group are known as code
loops. The first code loop is due to Parker, as discussed in \citep{Co}, and
the first systematic exposition of code loops can be found in \citep{Gr}.

Let $Q$ be a code loop, $Z\le Z(Q)$, $|Z|=2$, $V=Q/Z$ elementary abelian. For
$x$, $y\in Q$, denote by $[x,y]$ the \emph{commutator} of $x$, $y$, that is,
the unique element $u$ of $Q$ such that $xy = (yx)u$. For $x$, $y$, $z\in Q$,
denote by $[x,y,z]$ the \emph{associator} of $x$, $y$, $z$, that is, the unique
element $v$ of $Q$ such that $(xy)z = (x(yz))v$.

The three maps
\begin{align*}
    &P:Q\to Q,\, x\mapsto x^2 \textrm{ (power map),}\\
    &C:Q\times Q\to Q,\, (x,y)\mapsto [x,y] \textrm{ (commutator map),}\\
    &A:Q\times Q\times Q\to Q,\, (x,y,z)\mapsto [x,y,z] \textrm{ (associator map)}
\end{align*}
can in fact be considered as maps
\begin{displaymath}
    P:V\to Z,\quad C:V\times V\to Z,\quad A:V\times V\times V\to Z,
\end{displaymath}
and therefore identified with forms from the vector space $V$ to the field
$\gf{2}$.

The three forms are related by combinatorial polarization: $A$ is a trilinear
alternating form,
\begin{displaymath}
    C(x,y) = P(x+y) - P(x) - P(y),
\end{displaymath}
and
\begin{multline*}
    A(x,y,z) = C(x+y,z) - C(x,z) - C(y,z)\\
    = P(x+y+z) - P(x+y) - P(x+z) - P(y+z) + P(x) + P(y) + P(z).
\end{multline*}

We digress for a while to give more details on combinatorial polarization. The
material of Subsection \ref{Ss:CombPol} is taken from \citep{DrVo}. See
\citep{Wa} for an introduction to combinatorial polarization.

\subsection{Combinatorial polarization}\label{Ss:CombPol}

Let $V$ be a vector space over the $p$-element field $\gf{p}$, $p$ a prime. For
a map $\alpha:V\to \gf{p}$ and $n>1$ define $\alpha_n:V^n\to \gf{p}$ by
\begin{equation}\label{Eq:CombPol}
    \alpha_n(u_1,\dots,u_n) = \sum_{\{i_1,\dots,i_m\}\subseteq\{1,\dots,n\}}
        (-1)^{n-m}\alpha(u_{i_1}+\cdots+u_{i_m}),
\end{equation}
where $\alpha(\emptyset)=0$. Then $\alpha_n$ is clearly a symmetric form,
called the \emph{$n$th derived form of} $\alpha$. We say that
$\alpha=\alpha_1$, $\alpha_2$, $\alpha_3$, $\dots$ are \emph{related by
polarization}.

The \emph{combinatorial degree} of $\alpha:V\to \gf{p}$ is the largest integer
$n$ such that $\alpha_n\ne 0$ and $\alpha_m=0$ for every $m>n$, if it exists.

The defining identity (\ref{Eq:CombPol}) is equivalent to the recurrence
relation
\begin{multline}\label{Eq:Recurrence}
    \alpha_n(u,v,w_3,\dots,w_n)\\
     =\alpha_{n-1}(u+v,w_3,\dots,w_n)
     -\alpha_{n-1}(u,w_3,\dots,w_n)-\alpha_{n-1}(v,w_3,\dots,w_n).
\end{multline}
We see from (\ref{Eq:Recurrence}) that the combinatorial degree of $\alpha$ is
equal to $n$ if and only if $\alpha_n\ne 0$ is a symmetric $n$-linear form
(since $\gf{p}$ is a prime field). Moreover, an easy induction proves:

\begin{lem}\label{Lm:Arg0}
Let $V$ be a vector space over $\gf{p}$ and $\alpha:V\to\gf{p}$ a map
satisfying $\alpha(0)=0$. Then $\alpha_n(0,u_2,\dots,u_n)=0$ for every $u_2$,
$\dots$, $u_n\in V$.
\end{lem}

\begin{prop}\label{Pr:KnowMap}
Let $V$ be a vector space over $\gf{p}$ with basis $B=\{e_1$, $\dots$, $e_d\}$.
Let $\alpha:V\to\gf{p}$ be a map of combinatorial degree $n$. Then the
following conditions are equivalent:
\begin{enumerate}
\item[(i)] $\alpha(u_1)$, $\alpha_2(u_1,u_2)$, $\dots$,
$\alpha_n(u_1,\dots,u_n)$ are known for every $u_1$, $\dots$, $u_n\in V$,

\item[(ii)] $\alpha(u_1)$, $\alpha_2(u_1,u_2)$, $\dots$,
$\alpha_n(u_1,\dots,u_n)$ are known for every $u_1$, $\dots$, $u_n\in B$.
\end{enumerate}
\end{prop}

\begin{proof}
Clearly, (i) implies (ii). Assume that (ii) holds. Then
$\alpha_n(u_1,\dots,u_n)$ is known for every $u_1$, $\dots$, $u_n\in V$ since
$\alpha_n$ is $n$-linear and by Lemma \ref{Lm:Arg0}.

Assume that $k>0$ and $\alpha_{k+1}(u_1,\dots,u_{k+1})$ is known for every
$u_1$, $\dots$, $u_{k+1}\in V$. Write $u_i=\sum_{j=1}^d u_{ij}e_j$, and let
$||u_i||=\sum_{j=1}^n u_{ij}$, where the \emph{norm} is calculated in $\mathbb
Z$, not in $\gf{p}$. We show that $\alpha_k(u_1,\dots,u_k)$ is known for every
$u_1$, $\dots$, $u_k\in V$ by induction on $\sum_{i=1}^k ||u_i||$.

If for every $1\le i\le k$ we have $||u_i||\le 1$, then
$\alpha_k(u_1,\dots,u_k)$ is known by (ii) and by Lemma \ref{Lm:Arg0}.
Otherwise we can assume that either $u_1$ has two nonzero coefficients, or that
$u_1$ has a nonzero coefficient larger than $1$. In any case, upon writing
$u_1$ as a sum of two vectors of smaller norm, we are done by the induction
hypothesis and by the recurrence relation (\ref{Eq:Recurrence}) for
$\alpha_{k+1}$ and $\alpha_k$.
\end{proof}

\subsection{Code loops up to isomorphism}

Let us return to code loops of order $64$. In the notation of derived forms, we
have $C = P_2$ and $A=P_3$. The code loop $Q$ is determined by the three forms
$P$, $C$, $A$, and hence, by Proposition \ref{Pr:KnowMap}, by the values
\begin{equation}\label{Eq:ToKnow}
    P(e_i),\quad C(e_i,e_j),\quad A(e_i,e_j,e_k),
\end{equation}
where $\{e_1,\dots,e_5\}$ is a basis of $V$, and where $i<j<k$.

Moreover, Aschbacher shows in \citep{As} that two code loops $Q$, $Q'$ with
associated triples $(P, C, A)$ and $(P',C',A')$ are isomorphic if and only if
$(P,C,A)$ and $(P',C',A')$ are \emph{equivalent}, i.e., there is $\psi\in
\mathrm{GL}(V)$ such that
\begin{displaymath}
    P(u) = P'(\psi(u)),\quad
    C(u,v) = C'(\psi(u),\psi(v)),\quad
    A(u,v,w) = A'(\psi(u),\psi(v),\psi(w))
\end{displaymath}
for every $u$, $v$, $w\in V$.

In order to construct all nonassociative code loops of order $64$, we must
first find all triples $(P,C,A)$ up to equivalence, where $P$ has combinatorial
degree $3$. (If the combinatorial degree of $P$ is less than $3$ then the
associator map $A=P_3$ is trivial and hence the resulting loop is a group.) For
the convenience of the reader, we give formulae for evaluating $A$, $C$ and $P$
on $V$ based only on (\ref{Eq:ToKnow}) and on the symmetry of the three forms:
\begin{displaymath}
    A(\sum_i x_ie_i, \sum_j y_je_j, \sum_k z_ke_k) =
        \sum_{i,j,k} x_iy_jz_j A(e_i,e_j,e_k),
\end{displaymath}
\begin{multline*}
    C(\sum_i x_ie_i, \sum_j y_je_j) =
    \sum_{i,j} x_iy_j C(e_i,e_j) + \sum_{k}\sum_{i<j} x_ix_jy_k
        A(e_i,e_j,e_k)\\ + \sum_i\sum_{j<k} x_iy_jy_k A(e_i,e_j,e_k),
\end{multline*}
and
\begin{displaymath}
    P(\sum_i x_ie_i) =
        \sum_i x_iP(e_i) + \sum_{i<j} x_ix_j C(e_i,e_j)
        + \sum_{i<j<k} x_ix_jx_k A(e_i,e_j,e_k),
\end{displaymath}
where all running indices range from $1$ to $5$.

A linear transformation $M = (m_{ij})\in \mathrm{GL}(V)$ turns the triple
$(P,C,A)$ into a triple $(P^M, C^M, A^M)$ according to
\begin{displaymath}
    A^M(e_i,e_j,e_k) = \sum_{u,v,w} m_{iu}m_{jv}m_{kw} A(e_u,e_v,e_w),
\end{displaymath}
\begin{multline*}
    C^M(e_i,e_j) =
        \sum_{u,v} m_{iu}m_{jv} C(e_u,e_v)
        + \sum_w \sum_{u<v} m_{iu}m_{iv}m_{jw} A(e_u,e_v,e_w)\\
        + \sum_u \sum_{v<w} m_{iu}m_{jv}m_{jw} A(e_u,e_v,e_w),
\end{multline*}
and
\begin{displaymath}
    P^M(e_i) = \sum_u m_{iu}P(e_u) + \sum_{u<v} m_{iu}m_{iv} C(e_u,e_v)
        + \sum_{u<v<w} m_{iu}m_{iv}m_{iw} A(e_u,e_v,e_w).
\end{displaymath}

A computer search based on the above formulae produced $80$ nonequivalent
triples $(P,C,A)$. However, it is conceivable (and, in fact, it does happen)
that some of the associated code loops were already obtained earlier in the
search as extensions $E(K,\gf{2},f)$ for some Moufang loop $K$ of order $32$
that is not elementary abelian. It is therefore necessary to construct the $80$
code loops explicitly and test them for isomorphism against the previously
found loops.

There are several ways in which the code loop $Q$ can be recovered from the
triple $(P,C,A)$. The first, iterative construction is due to \citep{Gr},
another construction (also iterative) using twisted products was given in
\citep{Hs}, and the most recent construction is due to the first author
\citep{Na}. In the latter paper, Nagy shows that there is a one-to-one
correspondence between nonequivalent triples $(P,C,A)$, a certain class of
non-S-isomorphic groups with triality, and code loops. The correspondence of
\citep{Na} is constructive, and we used it to obtain a concrete description of
the $80$ code loops.

\section{Conclusion of the search}

We have by now obtained all Moufang loops of order $64$  (resp. $81$) by
producing all central extensions $E(K,\gf{2},f)$ (resp. $E(K,\gf{3},f)$), where
$K$ is an at least three-generated Moufang loop of order $32$ (resp. $27$) and
$f:K\times K\to\gf{2}$ (resp. $f:K\times K\to\gf{3}$) is a Moufang cocycle
modulo coboundaries modulo the action of $\mathrm{Aut}(K)$.

Upon sorting the loops up to isomorphism and discarding groups (which arise
when $K$ is associative and $f$ is a group cocycle, see Remark \ref{Rm:Group}),
we have found 4262 nonassociative Moufang loops of order $64$ and $5$
nonassociative Moufang loops of order $81$. This finishes the proof of Theorem
\ref{Th:Main} and the enumerative part of Theorem \ref{Th:Main81}. One can then
check, for instance in the \textsf{LOOPS} package, that $2$ of the $5$
nonassociative Moufang loops of order $81$ are commutative, and that the
remaining $3$ loops are isotopes of the commutative ones. We have proved
Theorem \ref{Th:Main81}.

The 4262 nonassociative Moufang loops of order $64$ and the $5$ nonassociative
Moufang loops of order $81$ are available electronically in the latest version
of \textsf{LOOPS}. The command \texttt{MoufangLoop(n,m)} retrieves the $m$th
Moufang loop of order $n$ from the database.

\begin{rem} Based on the discussion in Subsection $\ref{Ss:HistBackground}$,
we see that the classification of nonassociative Moufang loops of order $p^4$
is now complete. Moreover, by \citep{NagyValsecchi}, the only case missing in
the classification of nonassociative Moufang loops of order $p^5$ is $3^5=243$.
The tools used here fall just short of covering this case. (The associated
systems of linear equations can be solved but the isomorphism problem is too
difficult for the methods present in the \textsf{LOOPS} package. We believe it
could be solved using a specialized algorithm for Moufang $3$-loops.)
\end{rem}

\begin{rem}\label{Rm:Group}
Given a loop $K$ and an abelian group $A$ we say that $f:K\times K\to A$ is a
\emph{group cocycle} if
\begin{displaymath}
    f(xy,z) + f(x,y) = f(x,yz) + f(y,z)
\end{displaymath}
holds for every $x$, $y$, $z\in K$. When $K$ is a group and $f$ is a group
cocycle, $E(K,A,f)$ is a group. Group cocycles form a subspace of $\mcoc{K}$
containing $\cob{K}$.

The reader might wonder why we did not take advantage of group cocycles in the
search to further cut the dimension of the complement $\comp{K}$. (For
illustration, when $K$ is the elementary abelian group of order $32$, the
subspace of group cocycles has dimension $41$, compared to $\dim(\cob{K})=26$.)
The difficulty is that two (Moufang) cocycles that differ by a group cocycle do
not necessarily yield isomorphic loops.
\end{rem}

\subsection{Technical information}

We conclude the paper with some technical information concerning the search.

The calculations have been carried twice, on two different computers, and with
slightly different algorithms. We worked within the \textsf{GAP} \citep{GAP}
package \textsf{LOOPS} \citep{LOOPS}. The \textsf{GAP} code for all algorithms
specific to this paper can be downloaded at http://www.math.du.edu/\~{}petr in
section ``Publications''. The code is well commented and contains additional
details about the search not provided here.

The total running time of the program on a $2$ gigahertz PC with Windows XP
operating system was about $3$ hours, out of which about $1$ minute was devoted
to the case $n=81$, and about $15$ minutes to the elementary abelian case for
$n=64$.

For each Moufang loop $K$ of order $32$ that is not two-generated the program
returned a list of IDs of pairwise nonisomorphic nonassociative Moufang loops
of order $64$ that are central extensions of $K$ by $\gf{2}$. These $102$ lists
contained $11434$ IDs, with $4262$ unique IDs, and with maximal multiplicity of
an ID equal to $7$. Precisely $64$ out of the $80$ code loops of order $64$
cannot be obtained as central extensions of any other Moufang loop of order
$32$ than the elementary abelian group.


\bibliographystyle{elsart-harv}

\begin{thebibliography}{}

\bibitem[Aschbacher(1994)]{As}
Aschbacher, M., 1994. \textit{Sporadic groups}, Cambridge Tracts in Mathematics
\textbf{104}, Cambridge University Press, Cambridge.

\bibitem[Bol(1937)]{Bol}
Bol, G., 1937. \emph{Gewebe und Gruppen}, Math. Ann. \textbf{114}, 414--431.

\bibitem[Bruck(1971)]{Br}
Bruck, R.~H., 1971. A Survey of Binary Systems, third printing, corrected,
\emph{Ergebnisse der Mathematik und Ihrer Grenzgebiete}, New series, Volume
\textbf{20}, Springer-Verlag, New York-Heidelberg-Berlin.

\bibitem[Chein(1978)]{Ch}
Chein, O., 1978. \emph{Moufang loops of small order}, Memoirs of the American
Mathematical Society, Volume \textbf{13}, Issue 1, Number \textbf{197}.

\bibitem[Chein and Rajah(2003)]{CheinRajah}
Chein, O., Rajah, A., 2003. \emph{Possible orders of nonassociative Moufang
loops}, Comment. Math. Univ. Carolin. \textbf{41}, \textbf{2} (April 2003),
237--244.

\bibitem[Conway(1985)]{Co}
Conway, J.~H., 1985. \emph{A simple construction for the Fisher-Griess monster
group}, Invent. Math. \textbf{79}, no. \textbf{3}, 513--540.

\bibitem[Dr\'apal(2003)]{Dr}
Dr\'apal, A., 2003. \emph{Cyclic and dihedral constructions of even order},
Comment.\ Math.\ Univ.\ Carolin. \textbf{44}, \textbf{4}, 593--614.

\bibitem[Dr\'apal and Vojt\v{e}chovsk\'y(2007)]{DrVo}
Dr\'apal, A., Vojt\v{e}chovsk\'y, P., 2007. \emph{Symmetric multilinear forms
and polarization of polynomials}, submitted.

\bibitem[Hsu(2000)]{Hs}
Hsu, T., 2000. \emph{Moufang loops of class $2$ and cubic forms}, Math. Proc.
Cambridge Philos. Soc. \textbf{128}, 197--222.

\bibitem[\textsf{GAP}(2006)]{GAP}
The \textsf{GAP} Group, \textsf{GAP} --- Groups, Algorithms, and Programming,
Version 4.4; Aachen, St Andrews (2006). (Visit
http://www-gap.dcs.st-and.ac.uk/\~{}gap).

\bibitem[Glauberman(1968)]{Glauberman}
Glauberman, G., 1968. \emph{On loops of odd order II}, J.~Algebra \textbf{8},
393--414.

\bibitem[Glauberman and Wright(1968)]{GW}
Glauberman, G., Wright, C.~R.~B., 1968. \emph{Nilpotence of finite Moufang
$2$-loops}, J.~Algebra \textbf{8}, 415--417.

\bibitem[Goodaire, May and Raman(1999)]{Go}
Goodaire, E.~G., May, S., Raman, M., 1999. The Moufang Loops of Order less than
$64$, Nova Science Publishers.

\bibitem[Griess(1986)]{Gr}
Griess,~Jr., R.~L., 1986. \emph{Code Loops}, J.~Algebra \textbf{100},
$224$--$234$.

\bibitem[Kepka and N\v{e}mec(1981)]{KN}
Kepka, T., N\v{e}mec, P., \emph{Commutative Moufang loops and distributive
groupoids of small orders}, Czechoslovak Math. J. \textbf{31} (\textbf{106})
(1981), no. \textbf{4}, 633--669.

\bibitem[Leong and Rajah(1997)]{LeongRajah}
Leong, F., Rajah, A., 1997. \emph{Moufang loops of odd order $p^\alpha
q_1^2\cdots q_n^2r_1\cdots r_m$}, J.~Algebra \textbf{190}, 474--486.

\bibitem[Nagy(2007)]{Na}
Nagy, G.~P., 2007. \emph{Direct construction of code loops}, to appear in
Discrete Mathematics.

\bibitem[Nagy and Valsecchi(2007)]{NagyValsecchi}
Nagy G.~P., Valsecchi, M., 2007. \emph{On nilpotent Moufang loops with central
associators}, J.~Algebra \textbf{307} (2007), 547--564.

\bibitem[Nagy and Vojt\v{e}chovsk\'y(2007)]{LOOPS}
Nagy G.~P., Vojt\v{e}chovsk\'y, P., 2007. \textsf{LOOPS}: \emph{Computing with
quasigroups and loops in} \textsf{GAP}, download at
\texttt{http://www.math.du.edu/loops}.

\bibitem[Pflugfelder(1990)]{Pf}
Pflugfelder, H.~O., 1990. Quasigroups and Loops: Introduction, \emph{Sigma
series in pure mathematics} {\bf 7}, Heldermann Verlag Berlin.

\bibitem[Rajah(2001)]{Rajah}
Rajah, A., 2001. \emph{Moufang loops of odd order $pq\sp 3$}, J.~Algebra
\textbf{235}, no. \textbf{1}, 66--93.

\bibitem[Vojt\v{e}chovsk\'y(2006)]{Vo}
Vojt\v{e}chovsk\'y, P., 2006. \emph{Toward the classification of Moufang loops
of order $64$}, European J. Combin. \textbf{27}, issue \textbf{3} (April 2006),
444--460.

\bibitem[Ward(1979)]{Wa}
Ward, H.~N., 1979. \emph{Combinatorial Polarization}, Discrete Mathematics
\textbf{26}, 186--197.

\bibitem[Wright(1965)]{Wright}
Wright, C.~R.~B., 1965. \emph{Nilpotency conditions for finite loops}, Illinois
J.~Math. \textbf{9}, 399--409.

\end{thebibliography}

\end{document}